\documentclass[letterpaper]{amsart}

\usepackage{amsmath,amssymb,amsfonts,mathrsfs,diagrams,color}

\begin{document}
\title{the associated map of the nonabelian Gauss-Manin connection}
\author{Ting Chen}
\maketitle

\begin{abstract}
The Gauss-Manin connection for nonabelian cohomology spaces is the isomonodromy flow. We write down explicitly the vector fields of the isomonodromy flow and calculate its induced vector fields on the associated graded space of the nonabelian Hogde filtration. The result turns out to be intimately related to the quadratic part of the Hitchin map.
\end{abstract}

\section{introduction}

The variation of Hodge structures for families of complex K\"ahler manifolds has been a much studied subject. Let $\pi :\mathscr{X}\to S$ be a proper holomorphic submersion of connected complex manifolds. Ehresmann's Lemma says that it is a locally trivial fiber bundle with respect to its underlying differentiable structure. In particular all the fibers of $\pi$ are diffeomorphic. So if $s\in S$ and $X_s$ the fiber of $\pi$ over $s$, $\mathscr{X}\to S$ can be viewed as a variation over $S$ of complex structures on the underlying differentiable manifold of $X_s$. 

Let $\pi ^k:\mathscr{V}\to S$ be the corresponding vector bundle of cohomologies whose fiber at $s\in S$ is $H^k(X_s,\mathbb{C})$, $k\in\mathbb{N}$. Since $\mathscr{X}\to S$ is locally trivial differentiably (and therefore topologically), there is an induced local identification of fibers of $\mathscr{V}\to S$. In another word, there is a flat connection on the vector bundle $\mathscr{V}\to S$. This connection is called the Gauss-Manin connection for the cohomologies of the family of complex K\"ahler manifolds $\mathscr{X}\to S$.

From Hodge theory we know there is a natural Hodge filtration on the vector bundle $\mathscr{V}\to S$: $\mathscr{V}=F^0\supset F^1\supset F^2\ldots\supset F^k$. Let $\nabla$ be the Gauss-Manin connection, Griffiths transversality theorem says that $\nabla (F^p)\subset F^{p-1}\otimes\Omega^1_S$, $1\leq p\leq k$. So if $gr\mathscr{V}$ is the associated graded vector bundle of the filtered bundle $\mathscr{V}$, then the induced map $gr\nabla$ of $\nabla$ on $gr\mathscr{V}$ will be $\mathscr{O}_S$-linear. In fact, $gr\nabla$ is equal to a certain Kodaira-Spencer map\cite{Gr}. We call $gr\nabla$ the \emph{associated map of the Gauss-Manin connection}.
 

The above has a nonabelian analogue. Let $G$ be a complex algebraic group, $X$ a smooth algebraic curve over $\mathbb{C}$ of genus $g$. Let $Conn_X$ be the moduli space of principal $G$-bundles over $X$ equipped with a flat connection. If we denote as $H^1(X,G)$ the first \v Cech cohomology of $X$ with coefficients the constant sheaf in $G$, then $Conn_X$ can be naturally identified with $H^1(X,G)$, by considering the gluing data of flat $G$-bundles. Since the group $G$ can be nonabelian, we call $Conn_X$ the nonabelian cohomology space of $X$.

Let $\mathscr{M}_g$ be the moduli space of genus $g$ complex algebraic curves. The universal curve $\mathscr{X}\to\mathscr{M}_g$ is (roughly) a variation of complex structures of the underlying real surface, and the universal moduli space of connections $\mathscr{C}onn\to\mathscr{M}_g$ is the corresponding bundle of nonabelian cohomologies. For the same reason as before there is a Gauss-Manin connection on the bundle $\mathscr{C}onn\to\mathscr{M}_g$. The local trivialization that defines it is often called the isomonodromy deformation, or the isomonodromy flow of $\mathscr{C}onn$ over  $\mathscr{M}_g$.

There is also a nonabelian analogue of Hodge filtration which was determined by Carlos Simpson\cite{Si}, using a generalized definition of filtration of spaces. A vector space with filtration is equivalent, by the Rees construction, to a locally free sheaf over $\mathbb{C}$ with a $\mathbb{C}^\ast$-action, and with the fiber over $1$ being the vector space itself. To define the nonabelian Hodge ``filtration'' on the space $\mathscr{C}onn$ therefore, it would be reasonable to find a family of spaces over $\mathbb{C}$ whose fiber over $1$ is $\mathscr{C}onn$, together with a $\mathbb{C}^\ast$-action on the family. The way to do it in this case is to introduce the notion of $\lambda$-connections on a principal $G$-bundle on $X$, for any $\lambda\in\mathbb{C}$. It is a generalization of the notion of connections on a $G$-bundle. In particular a $1$-connection is an ordinary connection, and a $0$-connection is a so called \emph{Higgs field}, which is an object of much interest to people in complex geometry and high energy physics. The moduli space of principal $G$-bundles over $X$ together with a Higgs field is called the Higgs moduli space over $X$, and denoted as $Higgs_X$. Simpson's definition of nonabelian Hodge filtration immediately implies that the associated graded space of $Conn_X$ is $Higgs_X$. Then a question arises: what is the associated map of the nonabelian Gauss-Manin connection on the associated graded space? The answer is: it is a lifting\footnote{Here the word lifting has a slightly more general meaning: it means a map of tangent vectors in the opposite direction of the pushforward, without requiring its composition with pushforward being identity. In fact, this lifting here composed with pushforward is zero.} of tangent vectors on the relative Higgs moduli space $\mathscr{H}iggs\to\mathscr{M}_g$. On the other hand there is a well-known Hitchin map from $Higgs_X$ to some vector spaces. The quadratic part of the Hitchin maps also induce a lifting of tangent vectors on $\mathscr{H}iggs\to\mathscr{M}_g$. The fact that the two liftings agree is the content of our theorem.

\newtheorem{thm0}{Theorem}[section]
\begin{thm0}
\label{thm0:main0}
The lifting of tangent vectors on $\mathscr{H}iggs\to\mathscr{M}_g$ representing the associated map of the nonabelian Gauss-Manin connection is equal up to a constant multiple to the lifting of tangent vectors induced from the quadratic Hitchin map.
\end{thm0}

Closely related results have been obtained in \cite{BzF}, where the authors apply localization for vertex algebras to the Segal-Sugawara construction of an internal action of the Virasoro algebra on affine Kac-Moody algebras to lift twisted differential operators from the moduli of curves to the moduli of curves with bundles. Their construction gives a uniform approach to several phenomena describing the geometry of the moduli spaces of bundles over varying curves, including a Hamiltonian description of the isomonodromy equations in terms of the quadratic part of Hitchin's system. Our result and proof are much more elementary, avoiding the need for the vertex algebra machinery.

The organization of the paper is as follows. In section 2 we give a detailed definition of all the objects concerned and a precise statement of the theorem. The rest of the sections are devoted to the proof. In section 3 we recall the definition of Atiyah bundles and some of its properties that will be useful in the proof. In section 4 we use deformation theory to write the tangent spaces to $\mathscr{C}onn$ as certain hypercohomology spaces. Section 5 gives an explicit description of the lifting of tangent vector on $\mathscr{C}onn\to\mathscr{M}_g$ given by the isomonodromy flow. Section 6 extend the isomonodromy lifting to the moduli space of $\lambda$-connections for any $\lambda\neq 0$. Finally section 7 takes the limit of the lifting at $\lambda =0$, which is precisely the associated map of the nonabelian Gauss-Manin connection, and shows that it is equal up to a constant to the quadratic Hitchin lifting of tangent vectors. 

I would like to thank my advisor Ron Donagi for introducing me to the subject and for many invaluable discussions.


\section{definitions and statement of the theorem}

All objects and morphisms in this paper will be algebraic over $\mathbb{C}$, unless otherwise mentioned. 

\subsection{Moduli space of connections and isomonodromy flow}

Let $g$ be a natural number greater or equal to 2, so that a generic curve of genus $g$ has no automorphisms. The moduli space of all genus $g$ curves is a smooth Deligne-Mumford stack, but if we restrict to the curves that has no automorphisms, the moduli space is actually a smooth scheme. Let $\mathscr{M}_g$ be this scheme. In this paper we will ignore all the special loci of the moduli spaces (as explained below) and focus on local behaviors around generic points.

Let $G$ be a semisimple Lie group, $X$ a smooth curve of genus $g$. Let $Bun_X$ be the coarse moduli space of regular stable $G$-bundles on $X$. $Bun_X$ is also a smooth scheme\cite{Mu}. The total space of the cotangent bundle $T^\ast Bun_X$ is an open subscheme of the Higgs moduli space over $X$ \cite{Hi}. However since we are only concerned with generic situations, we will use $Higgs_X$ to denote the open subscheme $T^\ast Bun_X$.

Let $Conn_X$ be the moduli space of pairs $(P,\nabla)$, where $P$ is a stable $G$-bundle on $X$, and $\nabla$ is a connection on $P$. $\nabla$ is necessarily flat as the dimension of $X$ is equal to 1. $Conn_X$ is an affine bundle on $Bun_X$ whose fiber over $P\in Bun_X$ is a torsor for $T_P^\ast Bun_X$. So it is also a smooth scheme.

Let $\mathscr{C}onn\to \mathscr{M}_g$ be the relative moduli space of pairs whose fiber at $X\in\mathscr{M}_g$ is $Conn_X$. Let $Irrep_X$ be the space of all irreducible group homomorphisms $\pi _1(X)\to G$, $Irrep_X$ is a smooth scheme\cite{IIS}. There is also the relative space $\mathscr{I}rrep\to \mathscr{M}_g$. The Riemann-Hilbert correspondence $RH: Conn_X\to Irrep_X$ taking a flat connection to its monodromy is an analytic(and therefore differentiable) inclusion. Let $S\subset\mathscr{M}_g$ be a small neighborhood of $X$ in analytic topology. By Ehresmann's Lemma the family of curves over $S$ is a trivial family with respect to the differentiable structure. This implies that the restriction of $\mathscr{I}rrep$ over $S$ is a differentiable trivial family. The trivial sections or trivial flows induce a flow on the restriction of $\mathscr{C}onn$ over $S$, by the Riemann-Hilbert correspondence. This flow is called the isomonodromy flow of $\mathscr{C}onn$ over $\mathscr{M}_g$.

\subsection{$\lambda$-connections and nonabelian Hodge filtration}

As explained in the last section, $Conn_X$ is the nonabelian cohomology space of $X$ with in coefficient $G$, and the isomonodromy flow on $\mathscr{C}onn\to \mathscr{M}_g$ is the nonabelian Gauss-Manin connection. To define Hodge filtration on $Conn_X$ one need to generalize the definition of a filtration. A filtration on a vector space $V$ is equivalent, by the Rees construction\cite{Hi}, to a locally free sheaf $W$ on $\mathbb{C}$ whose fiber at $1\in\mathbb{C}$ is isomorphic to $V$, together with a $\mathbb{C}^\ast$-action on $W$ compatible with the usual $\mathbb{C}^\ast$-action on $\mathbb{C}$. The fiber of $W$ at $0\in\mathbb{C}$ will be isomorphic to the associated graded vector space of $V$.

This sheaf definition of filtrations can be generalized in an obvious way to define filtrations on a space that is not a vector space. In our case the space is $Conn_X$, and its Hodge filtration is constructed as follows. $Conn_X$ parametrizes pairs ($P$,$\nabla$). Let $P$ also denote the sheaf of sections of $P$ on $X$, $adP$ be the adjoint bundle of $P$ as well as the sheaf of its sections, and $\mathfrak{g}$ the Lie algebra of $G$. A connection $\nabla$ is a map of sheaves
\[
\nabla :P \to adP\otimes\Omega _X^1
\]
that after choosing local coordinates for $X$ and local trivialization for $P$ can be written as 
\[
(\frac{\partial}{\partial x} +[A(x),\ ])\otimes dx
\]
where $A(x)$ is a $\mathfrak{g}$-valued function and the bracket means the right multiplication action of $G$ on $\mathfrak{g}$. A $\lambda$-connection on $P$ is defined to be a map of sheaves $\nabla_\lambda :P \to adP\otimes\Omega _X^1$ that in local coordinates can be written as $(\lambda\frac{\partial}{\partial x} +[A(x),\ ])\otimes dx$. Let the moduli space of $\lambda$-connections be denoted as $\lambda Conn_X$. For $\lambda\neq 0$, $\nabla\leftrightarrow\lambda\cdot\nabla$ is a bijection between $Conn_X$ and $\lambda Conn_X$. For $\lambda =0$, the definition of a $0$-connection agrees with that of a Higgs field. So $0Conn_X$ is just $Higgs_X$. 

Let $\mathcal{T}_X$ be the moduli space of all $\lambda$-connections for all $\lambda\in\mathbb{C}$. There is a natural map $\mathcal{T}_X\to\mathbb{C}$ taking a $\lambda$-connection to $\lambda$, whose preimage at $1\in\mathbb{C}$ is $Conn_X$. In fact, Simpson showed that the nonabelian Hodge filtration of $Conn_X$ is precisely the sheaf of sections of this map, with the $\mathbb{C}^\ast$-action given by multiplication by $\lambda$ for $\lambda\in\mathbb{C}^\ast$\cite{Hi}. The $\mathbb{C}^\ast$-action is algebraic and induces an isomorphism of $Conn_X$ and $\lambda Conn_X$.

In the ordinary Hodge theory, if one uses the sheaf definition of filtrations, then the associated map of the Gauss-Manin connection is obtained as follows. Start with the Gauss-Manin connection on $\mathscr{V}\to S$, the local trivialization by the flat sections gives a lifting of tangent vectors
\[
L: T_s S \to T_v \mathscr{V}
\]
for $s\in S$ and $v\in\mathscr{V}$ s.t. $\pi ^k(v)=s$. The lifting $L$ is a spliting of $\pi ^k_\ast$, i.e. it satisfies
\[
\pi ^k_\ast \circ L = id_{T_s S}
\]
Let $\mathscr{W}\to\mathbb{C}$ be the sheaf associated to the Hodge filtration on $\mathscr{V}\to S$. The fiber of $\mathscr{W}$ at $1$ is $\mathscr{V}\to S$, and denote the fiber over $\lambda$ as $\pi ^k_\lambda :\mathscr{V}_\lambda\to S$. The action of $\lambda\in\mathbb{C}^\ast$ induces an isomorphism of $\mathscr{V}$ and $\mathscr{V}_\lambda$. So the local trivialization of $\mathscr{V}\to S$ induces a local trialization of $\mathscr{V}_\lambda\to S$ via this isomorphism. Let $L_\lambda: T_s S \to T_{v_\lambda} \mathscr{V}_\lambda$ be the induced lifting on $\mathscr{V}_\lambda\to S$ \emph{multiplied by $\lambda$}. $L_\lambda$ satisfies
\[
\pi ^k_{\lambda\ast} \circ L_\lambda =\lambda\cdot id_{T_s S}
\]
$L_\lambda$ is defined for all $\lambda\neq 0$. For a fixed vector $\vec{t}\in T_s S$, the images of $\vec{t}$ under all the $L_\lambda$, $\lambda\neq 0$ gives a vector field on the total space of $\mathscr{W}$ away from $\mathscr{V}_0$, which is the fiber over $0\in\mathbb{C}$. The continuous limit of that vector field on $\mathscr{V}_0$ exist, and therefore defines a lifting $L_0:T_s S\to T_{v_0}\mathscr{V}_0$ on $\mathscr{V}_0\to S$. $L_0$ satisfies 
\[
\pi ^k_{0\ast} \circ L_0 = 0
\]
i.e. the images of $\vec{t}\in T_s S$ under $L_0$ is a vectors field on the \emph{fiber} $V_{0,s}$ of $\mathscr{V}_0$ over $s$. This vector field is in fact linear and defines a linear map on $V_{0,s}$. Also $\mathscr{V}_0$ is identified with $gr\mathscr{V}$. From these we see $L_0$ really gives a vector bundle map $gr\mathscr{V}\to gr\mathscr{V}\otimes\Omega^1_S$, and that map is the associated map of the Gauss-Manin that we started with.

So in nonabelian Hodge theory, in order to calculate the associated map of the nonabelian Gauss-Manin connection, we will start with the lifting $L$ induced from the isomonodromy flow on $\mathscr{C}onn\to \mathscr{M}_g$(by a slight abuse of notation we will use the same notations for the liftings, the meaning should be clear from the context), and try to find the associated limit lifting $L_0$. Specifically, let $\mathscr{T}\to\mathscr{M}_g$ be the relative moduli space whose fiber at $X\in\mathscr{M}_g$ is $\mathcal{T}_X$. $\mathscr{T}$ maps to $\mathbb{C}$ and the fiber at $\lambda$ is the relative moduli space of $\lambda$-connections, which is denoted $\lambda\mathscr{C}onn$. There is clearly also a $\mathbb{C}^\ast$-action on $\mathscr{T}$ compatible with the $\mathbb{C}^\ast$-action on $\mathbb{C}$. Let $L_\lambda$ be analogously the lifting on $\lambda\mathscr{C}onn\to \mathscr{M}_g$ induced by the lifting $L$ via the $\mathbb{C}^\ast$-action and multiplied by $\lambda$. Then the limit lifting $L_0$ will be the associated map that we want to calculate. It will again be a vertical lifting, i.e. the images of $L_0$ will be vectors
tangent to the fibers $Higgs_X$ of $\mathscr{H}iggs\to\mathscr{M}_g$, $X\in\mathscr{M}_g$.  

\subsection{Quadratic Hitchin map and statement of the theorem}

$Higgs_X$ has a symplectic structure as it is equal to $T^\ast Bun_X$. Let $<\ ,\ >$ be the Killing form on the Lie algebra $\mathfrak{g}$ of $G$, the quadratic Hitchin map is 
\[
qh: Higgs_X \to H^0(X,\Omega ^{\otimes 2})
\]
\[
(P,\theta) \mapsto <\theta,\theta>
\]
where $\theta\in H^0(X,adP\otimes\Omega^1_X)$ is a 0-connection or a Higgs field. One can define a lifting of tangent vectors associated to $qh$
\[
L_{qh}: T_X\mathscr{M}_g \to T_{(P,\theta )}Higgs_X
\]
\[
f \mapsto H_{qh^\ast f}|_{(P,\theta )}
\] 
where $f\in T_X\mathscr{M}_g \cong H^1(X,TX)$ is viewed as a linear function on $H^0(X,\Omega ^{\otimes 2})$ by Serre duality, and $H_{qh^\ast f}$ is the Hamiltonian vector field of $qh^\ast f$ on $Higgs_X$.

The theorem can now be more precisely stated as:
\newtheorem{thm1}{Theorem}[section]
\begin{thm1}[precise version of Theorem~\ref{thm0:main0}]
\label{thm1:main1}
The limit lifting of tangent vectors $L_0$ associated to the isomonodromy lifting $L$ is equal to $\frac{1}{2}L_{qh}$.
\end{thm1}


\section{Atiyah bundles}

Before starting to prove the theorem, we recall here some facts about Atiyah bundles which will be used later. As before let $X$ be a smooth curve of genus $g$, $G$ a semisimple Lie group, $p:P\to X$ a principal $G$-bundle over $X$.

\subsection{Atiyah bundle and its sections}
\label{sec:abs}
Let $TP$ be the tangent bundle over $P$. $G$ acts on $P$ and has an induced action on $TP$. The action is free and compatible with the vector bundle structure of $TP\to P$, so the quotient will be a vector bundle $TP/G \to P/G=X$. This vector bundle  over $X$ is called the Atiyah bundle associated to $P$, and denoted as $A_P$.

In fact, $TP$ is isomorphic to the fiber product of $P$ and $A_P$ over $X$. So any section $t$ of $A_P$ over $X$ has a unique lift $\tilde{t}$ that makes the diagram commute
\newarrow{Dashto} {}{dash}{}{dash}>
\begin{diagram}
TP & \rTo^{/G} & A_P \\
\dTo\uDashto_{\tilde{t}} & & \dTo\uDashto_{t} \\
P & \rTo^{/G} & X \\
\end{diagram}

The lift $\tilde{t}$ can be viewed as a vector field on $P$ which is $G$-invariant. Conversely, any $G$-invariant vector field on $P$ defines a section $t$ in the quotient bundle. Therefore sections of $A_P$ over $X$ are the same as $G$-invariant vector fields on $P$.

\subsection{Atiyah sequence}
The sequence of tangent bundles associated to $P\to X$ is:
\[
0 \to T_{P/X}\to TP \to p^\ast TX \to 0
\]
$G$ acts on the sequence, and the quotient is
\[
0 \to adP \to A_P \to TX \to 0
\]
This quotient sequence is called the Atiyah sequence of $A_P$. We will denote the map $A_P \to TX$ also as $p_\ast$.

\subsection{Relation to connections}
\label{sec:rtc}
If $\nabla$ is a connection on $P$, then $\nabla$ must be flat since the dimension of $X$ is 1. So over a small open subset $U\subset X$, there is a natural trivialization of $P$ associated to $\nabla$ 
\[
\tau : U \times F \longrightarrow P|_U
\]
given by the flat sections of $\nabla$. Here F denotes a torsor for $G$.

The local trivialization gives a local section $\tilde{s}_U: p^\ast TU \to TP|_U$, which is the composition
\begin{equation}
p^\ast TU \xrightarrow{\tau ^{-1} _\ast} p_U^\ast TU \xrightarrow{(id,0)} p_U^\ast TU \oplus p_F^\ast TF \xrightarrow{\tau _\ast} TP|_U
\end{equation}
where $p_U$ and $p_F$ are the projections of $U\times F$ to $U$ and $F$.

Since $\tilde{s}_U$ is canonically associated to $\nabla$, so for two such open subsets $U,V$, $\tilde{s} _U$ and $\tilde{s} _V$ agree on their intersection. So there is a well-defined map $\tilde{s}: p^{\ast}TU \to TP$. Since $\tau$ is $G$-invariant and the map $(id,0)$ is obviously $G$-invariant, $\tilde{s}_U$ is $G$-invariant. So $\tilde{s}$ is $G$-invariant, and gives a map $s: TX \to A_P$. The map $(id,0)$ in the definition of $\tilde{s}_U$ implies that $s$ is a splitting of $p_\ast :A_P \to TX$, i.e. $p_\ast\circ s = id_{TX}$. We can also say that $s$ is a splitting of the Atiyah sequence.
\begin{diagram}[size=2em]
0 & \rTo & adP & \rTo & A_P & \pile{\rTo^{p_\ast} \\ \lDashto_s} & TX & \rTo & 0
\end{diagram}

To summarize, for any connection $\nabla$ on $P$ there is uniquely associated a splitting $s$ of the Atiyah sequence of $P$. $s$ is locally defined as the splitting $(id,0)$ with $P$(and therefore $A_P$) locally trivialized by $\nabla$.


\section{tangent spaces}

Now we start to prove the theorem. In this section we will identify the tangent spaces of $\mathscr{C}onn$ and more generally $\lambda\mathscr{C}onn$ as some hypercohomology spaces, so that we may write down the isomonodromy lifting $L$ and the extened liftings $L_\lambda$ explicitly in the next two sections.

The tangent space to a moduli space at a regular point is identified with the infinitesimal deformations of the object corresponding to that point. So we are really looking at infinitesimal deformations of the objects parametrized by $\mathscr{C}onn$, which are triples ($X$,$P$,$\nabla$). We start with deformations of pairs ($X$,$P$).

\subsection{Deformation of pairs}
\label{sec:dp}
From Deformation Theory, the following two propositions are well-known.

\newtheorem{tc}{Proposition}[subsection]
\begin{tc}
\label{tc}
The tangent space to $\mathscr{M}_g$ at a point $X$ is naturally isomorphic to $H^1(X,TX)$.
\end{tc}

\newtheorem{tb}[tc]{Proposition}
\begin{tb}
\label{tb}
The tangent space to $Bun_X$ at a point $P$ is naturally isomorphic to $H^1(X,adP)$.
\end{tb}

Let $\mathscr{B}un$ be the moduli space of pairs ($X$,$P$). We expect that generically the tangent space at a point ($X$,$P$) would satisfy
\[
0 \to H^1(X,adP) \to T_{(X,P)}\mathscr{B}un \to H^1(X,TX) \to 0
\]
On the other hand since the Atiyah sequence of $P$ $0 \to adP \to A_P \to TX \to 0$ induces
\[
0 \to H^1(X,adP) \to H^1(X,A_P) \to H^1(X,TX) \to 0
\]
It is natural to guess that

\newtheorem{prop3}[tc]{Proposition}
\begin{prop3}
\label{dp}
$T_{(X,P)}\mathscr{B}un$ is naturally isomorphic to $H^1(X,A_P)$.
\end{prop3}
\begin{proof} the proof is a combination of the usual proofs for Proposition~\ref{tc} and Proposition~\ref{tb}. Let $\{ U_i \}_{i\in I}$ be an \v Cech covering of $X$, $P_\epsilon \to X_\epsilon$ a family of principal $G$-bundles over $D_\epsilon =\mathbb{C}[\epsilon]/(\epsilon ^2)$, which restrict to $P \to X$ over the closed point. Over each $U_i$, let 
\[
\phi _i: P|_{U_i}\times D_\epsilon \to P_\epsilon|_{U_i}\ \ \ \ \ \         
(\phi _i^\vee : \mathscr{O}_{P|_{U_i}} \otimes \mathbb{C}[\epsilon]/(\epsilon ^2) \gets \mathscr{O}_{P_\epsilon|_{U_i}})
\]
be an isomorphism of $G$-bundles. So it is compatible with the $G$-actions and descends to an isomorphism
\[
\iota _i: U_i \times D_\epsilon \to X_\epsilon |_{U_i}\ \ \ \ \ \          
(\iota _i^\vee : \mathscr{O}_{U_i} \otimes \mathbb{C}[\epsilon]/(\epsilon ^2) \gets \mathscr{O}_{X_\epsilon|_{U_i}})
\]
Over $U_{ij}=U_i \cap U_j$, the transition functions are related as in the commutative diagram
\begin{diagram}
P|_{U_{ij}}\times D_\epsilon & \rTo^{\phi _j ^{-1} \circ \phi _i} & P|_{U_{ij}}\times D_\epsilon \\
\dTo_p & & \dTo_p \\
U_{ij} \times D_\epsilon & \rTo^{\iota _j ^{-1} \circ \iota _i} & U_{ij} \times D_\epsilon \\
\end{diagram}

Let $\xi _{ij} \in \Gamma (U_{ij},TX)$ be the vector field on $U_{ij}$ such that $(\iota _j ^{-1} \circ \iota _i)^\vee =Id+\epsilon\xi _{ij}$, and $\eta _{ij} \in \Gamma (P|_{U_{ij}},TP)$ be the vector field on $P|_{U_{ij}}$ such that $(\phi _j ^{-1} \circ \phi _i)^\vee =Id+\epsilon\eta _{ij}$. Because $\phi _i$ is $G$-invariant, $\eta _{ij}$ is $G$-invariant. So one can view it as $\eta _{ij} \in \Gamma (U_{ij},A_P)$. $(\eta _{ij})_{i,j\in I}$ form a \v Cech 1-cochain on $X$ with coefficients in $A_P$.

$(\eta _{ij})_{i,j\in I}$ is closed because it comes from transition functions $\phi _j ^{-1} \circ \phi _i$. Any closed cochain $(\eta _{ij})_{i,j\in I}$ comes from some $D_\epsilon$ family of pairs. Also for a fixed $D_\epsilon$ family of pairs, a different choice of $\phi _i$'s will result in a cocycle differing from $(\eta _{ij})_{i,j\in I}$ by an exact cocycle. And any exact cocycle is the result of different choices of $\phi _i$'s. Therefore the infinitesimal deformations of ($X$,$P$) are in natural correspondence with $H^1(X,A_P)$, which proves the proposition. 
\end{proof}

\subsection{Deformation of triples}
\label{sec:dt}
Now we come to the infinitesimal deformations of a triple ($X$,$P$,$\nabla$). First a notation related to the connection $\nabla$. As discussed in section~\ref{sec:rtc}, a connection $\nabla$ on $P$ is equivalent to a splitting of the Atiyah sequence 
\begin{diagram}[size=2em]
0 & \rTo & adP & \rTo & A_P & \pile{\rTo^{p_\ast} \\ \lDashto_s} & TX & \rTo & 0
\end{diagram}
Let $\hat{s}\in H^0(X,A_P\otimes \Omega _X^1)$ denote the global section associated to the splitting map $s$. We see that $\hat{s}\mapsto 1$ under the map $H^0(X,A_P\otimes \Omega _X^1) \to H^0(X,TX\otimes \Omega _X^1)\cong H^0(X,\mathscr{O}_X)$.

To find the deformation of the triple ($X$,$P$,$\nabla$), let ($X_\epsilon$,$P_\epsilon$,$\nabla_\epsilon$) be a family of triples over $D_\epsilon$ starting with it. Let $s_\epsilon$ be the family of sections corresponding to $\nabla_\epsilon$. As in the proof of Proposition~\ref{dp}, let $\{ U_i \}_{i\in I}$ again be an \v Cech covering of $X$, and $\phi _i$, $\iota _i$, $i\in I$ defined in the same way. Let $s_i: TU_i \to A_P|_{U_i}$ and $\sigma _i: TU_i \to adP|_{U_i}$ be sections such that the following diagram commutate:
\begin{diagram}
A_P|_{U_i} \times D_\epsilon & \rTo^{d\phi _i} & A_{P_\epsilon}|_{U_i} \\
\uDashto^{s_i+\epsilon\sigma _i} \dTo_{p_\ast} &  & \dTo^{p_\ast} \uDashto_{s_\epsilon |_{U_i}} \\
TU_i \times D_\epsilon & \rTo^{d\iota _i} & TX_\epsilon |_{U_i} \\
\end{diagram}
The target space of $\sigma _i$ is $adP$ instead of $A_P$, because $p_\ast \circ s=id$ for all $s$, so $\sigma _i$, being the derivative of $s$ (locally on $U_i$, under the trivialization of the family $\phi _i$), projects to $0$ under $p_\ast$.

A deformation of the triple should contain the information about the deformation of the pair ($X$,$P$) as well as the deformation of $\nabla$. So the data associated to the infinitesimal family ($X_\epsilon$,$P_\epsilon$,$\nabla_\epsilon$) should be the pair:
\[
(\eta_{ij})_{i,j\in I},(\sigma _i)_{i\in I}
\]
where $(\eta_{ij})_{i,j\in I}$ is defined in section~\ref{dp} and shown to characterize the deformation of the pair ($X$,$P$), and $(\sigma _i)_{i\in I}$ describe the deformation of $\nabla$.

The data $((\eta_{ij})_{i,j\in I},(\sigma _i)_{i\in I})$ looks like a 1-cocycle in defining the hypercohomology of some complex of sheaves. Recall that the tangent space to $Higgs_X$ at a point $(P,\theta)$ is $\mathbb{H}^1(X,adP\xrightarrow{[\ ,\theta]}adP\otimes\Omega^1_X)$. We will prove an analogous result about the tangent spaces to $\mathscr{C}onn$. 

On $U_{ij}$, the transition relations are expressed in the following diagram:
\begin{diagram}
A_P|_{U_{ij}} \times D_\epsilon & \rTo^{d(\phi _j ^{-1} \circ \phi _i)} & A_P|_{U_{ij}} \times D_\epsilon \\
\uDashto^{s_i+\epsilon\sigma _i} \dTo_{p_\ast} &  & \dTo^{p_\ast} \uDashto_{s_j+\epsilon\sigma _j} \\
TU_{ij} \times D_\epsilon & \rTo^{d(\iota _j ^{-1} \circ \iota _i)} & TU_{ij} \times D_\epsilon \\
\end{diagram}

Since $(\iota _j ^{-1} \circ \iota _i)^\vee =Id+\epsilon\xi _{ij}$ and $(\phi _j ^{-1} \circ \phi _i)^\vee =Id+\epsilon\eta _{ij}$, we can write down the two horizontal maps more explicitly. $\forall\ Y+\epsilon Y_1 \in TU_{ij} \times D_\epsilon$, its image $Y'+\epsilon Y_1'$ under $d(\iota _j ^{-1} \circ \iota _i)$ is determined by: for any function $f$ on $U_{ij}$,
\[
(Y'+\epsilon Y_1')(f)=(I+\epsilon\xi _{ij})(Y+\epsilon Y_1)(I-\epsilon\xi _{ij})(f)
\]
After simplification we get $Y'=Y,Y_1'=Y_1+[\xi _{ij},Y]$, where the bracket is the Lie bracket of vector fields on $U_{ij}$. Similarly $\forall\ Z+\epsilon Z_1 \in A_P|_{U_{ij}} \times D_\epsilon$ (by section~\ref{sec:abs} it can be viewed as a $G$-invariant vector field on $P|_{U_{ij}}$), we get $d(\phi _j ^{-1} \circ \phi _i)(Z+\epsilon Z_1)=Z+\epsilon (Z_1+[\eta _{ij},Z])$, where the bracket is the Lie bracket of ($G$-invariant) vector fields on $P|_{U_{ij}}$.

The diagram is commutative, i.e. $\forall\ Y+\epsilon Y_1 \in TU_{ij} \times D_\epsilon$
\[
d(\phi _j ^{-1} \circ \phi _i)\circ (s_i+\epsilon\sigma _i) (Y+\epsilon Y_1) = (s_j+\epsilon\sigma _j) \circ d(\iota _j ^{-1} \circ \iota _i) (Y+\epsilon Y_1) 
\]
After simplification we get
\[
s_i(Y)=s_j(Y)
\]
\begin{equation}
\label{midzero}
(\sigma _j - \sigma _i)(Y) = [\eta _{ij},s_i(Y)]-s_j([\xi _{ij},Y])
\end{equation}

So if we use $\hat{\sigma}_i \in H^0(X,adP\otimes \Omega _X^1)$ to denote the global section associated to $\sigma _i$, the pair 
\[
((\eta_{ij})_{i,j\in I},(\hat{\sigma}_i)_{i\in I})
\]
is a hyper \v Cech 1-cochain on $X$ with coefficients in 
\[
A_P\xrightarrow{[\ ,\hat{s}]}adP\otimes\Omega^1_X
\]
where the map $[\ ,\hat{s}]$ is defined as: if $\hat{s}=s'\otimes \omega$, where $s'\in H^0(X,A_P), \omega \in H^0(X,\Omega _X^1)$, then $[\ ,\hat{s}]:=[\ ,s']\otimes\omega-s'\otimes [p_\ast(\ ),\omega]$.

\newtheorem{prop4}[tc]{Proposition}
\begin{prop4}
\label{dt}
$T_{(X,P,\nabla)}\mathscr{C}onn$ is naturally isomorphic to $\mathbb{H}^1(X,A_P\xrightarrow{[\ ,\hat{s}]}adP\otimes\Omega^1_X)$.
\end{prop4}
\begin{proof}
To any $D_\epsilon$ family of triples ($X_\epsilon$,$P_\epsilon$,$\nabla_\epsilon$) is associated a hyper 1-cochain $((\eta_{ij})_{i,j\in I},(\hat{\sigma}_i)_{i\in I})$ by the above discussion. It is closed because of three facts: first, $(\eta _{ij})_{i,j\in I}$ is a closed \v Cech 1-cochain with coefficients in $A_P$ - it's closed again because it comes from the transition function $\phi _j ^{-1} \circ \phi _i$; second, because of \eqref{midzero}; third, the complex $A_P\xrightarrow{[\ ,\hat{s}]}adP\otimes\Omega^1_X$ has only two nonzero terms. These three facts imply that $((\eta_{ij})_{i,j\in I},(\hat{\sigma}_i)_{i\in I})$ is closed. Any closed hyper 1-cochain comes from some $D_\epsilon$ family of triples. Also for a fixed $D_\epsilon$ family of triples, a different choice of the $\phi _i$'s will result in a hyper cocycle differing from $((\eta_{ij})_{i,j\in I},(\hat{\sigma}_i)_{i\in I})$ by an exact hyper cocycle. And any exact hyper cocycle is the result of different choices of the $\phi _i$'s. Therefore the infinitesimal deformations of ($X$,$P$,$\nabla$) are in natural correspondence with $\mathbb{H}^1(X,A_P\xrightarrow{[\ ,\hat{s}]}adP\otimes\Omega^1_X)$, which is what we need to prove. 
\end{proof}

\subsection{Tangent spaces to $\lambda\mathscr{C}onn$}
Let $\lambda\in\mathbb{C}$ be a fixed complex number. For the moduli space $\lambda\mathscr{C}onn$ of triples ($X$,$P$,$\nabla_\lambda$) where $\nabla_\lambda$ is a $\lambda$-connection, the statement about its tangent spaces is completely analogous to that when $\lambda =1$. 

For a $\lambda$-connection $\nabla_\lambda$ on $P$, $\lambda\neq 0$, $\frac{1}{\lambda}\nabla_\lambda$ is an ordinary connection, therefore corresponds to a splitting $s_{\frac{1}{\lambda}\nabla_\lambda}$ of the Atiyah sequence of $P$. Let $s_\lambda =\lambda\cdot s_{\frac{1}{\lambda}\nabla_\lambda}$, so $s_\lambda$ is a ``$\lambda$-splitting'' of the Atiyah sequence of $P$, i.e. $p_\ast\circ s_\lambda =\lambda\cdot id_{TX}$. Therefore to any $\lambda$-connection $\nabla_\lambda$($\lambda\neq 0$) is associated a $\lambda$-splitting of the Atiyah bundle. Notice that this is true for $\lambda =0$ as well, as a 0-splitting of the Atiyah bundle of $P$ is exactly a Higgs field on $P$.

Let $\hat{s}_\lambda\in H^0(X,A_P\otimes \Omega _X^1)$ be the global section associated to $s_\lambda$, we see $\hat{s}_\lambda\mapsto \lambda$ under the map $H^0(X,A_P\otimes \Omega _X^1) \to H^0(X,TX\otimes \Omega _X^1)\cong H^0(X,\mathscr{O}_X)$. The arguments in the last subsection can be repeated with slight changes (replace 1 by $\lambda$ at appropriate places) to give the following statement.

\newtheorem{prop5}[tc]{Proposition}
\begin{prop5}
$T_{(X,P,\nabla_\lambda)}\lambda\mathscr{C}onn$ is naturally isomorphic to $\mathbb{H}^1(X,A_P\xrightarrow{[\ ,\hat{s}_\lambda]}adP\otimes\Omega^1_X)$, $\forall \lambda\in\mathbb{C}$ 
\end{prop5}

\subsection*{Remark}
When $\lambda =0$, the result agrees with the previous results about tangent spaces to the Higgs moduli space.


\section{isomonodromy vector field}

The nonabelian Gauss-Manin connection on $\mathscr{C}onn \to \mathscr{M}_g$ is the isomonodromy flow. The local trivialization of $\mathscr{C}onn \to \mathscr{M}_g$ given by the flow induces a lifting of tangent vectors $L:T_X\mathscr{M}_g \to T_{(X,P,\nabla)}\mathscr{C}onn$. We have identified these tangent spaces as (hyper)cohomology spaces in the last section, now we will write down the map $L$ as a map of cohomology spaces. We start with a useful fact about an isomonodromy family of connections.

\subsection{Universal connection of an isomonodromy family}
\label{sec:gfc}
In \cite{IIS} Inaba et al. constructed the moduli space of triples ($X$,$P$,$\nabla$), and a universal $G$-bundle on the universal curve with a universal connection. Though they did it for a special case(rank 2 parabolic vector bundle on $\mathbb{P}^1$ with 4 points), the more general case can be done similarly. The universal connection, when restricted to an isomonodromy family of triples, has the following important property.

\newtheorem*{prop10}{Proposition}
\begin{prop10}
If ($X_t$,$P_t$,$\nabla _t$) is an isomonodromy family of triples over a complex line $D=Spec(\mathbb{C}[t])$, then the restriction of the universal connection on $P_t$(viewed as a $G$-bundle over the total space of $X_t$) is flat.
\end{prop10}
\begin{proof}
If we only look at the underlying differentiable structure, the isomonodromy family over $D=\mathbb{C}[t]$ is a trivial family of triples. The trivial family structure gives a flat connection on $P_t$, which must be equal to the restriction of the universal connection on $P_t$ since they are equal on each fiber of the family.
\end{proof}

\subsection{Isomonodromy lifting of tangent vectors}
For $\forall \lambda\in\mathbb{C}$, let $\pi_\lambda$ be the projection:
\[
\pi_\lambda: \lambda\mathscr{C}onn \to \mathscr{M}_g
\]
\[
(X,P,\nabla_\lambda) \mapsto X 
\]

From the proof of Proposition~\ref{dp} and the discussions in front of Proposition~\ref{dt} it is not hard to see that the differential of $\pi_\lambda$
\begin{diagram}
T_{(X,P,\nabla_\lambda )}\lambda\mathscr{C}onn & & & \cong & & & \mathbb{H}^1(X,A_P\xrightarrow{[\ ,\hat{s}_\lambda ]}adP\otimes\Omega^1_X) \\
\dTo^{\pi_{\lambda\ast}} \\
T_X\mathscr{M}_g & \cong & & H^1(X,TX) & & \cong & \mathbb{H}^1(X,TX\to 0) \\
\end{diagram}
is induced from the map $(p_\ast,0)$ of complexes of sheaves
\begin{diagram}
(A_P & \rTo^{[\ ,\hat{s}_\lambda]} & adP\otimes\Omega ^1_X) \\
\dTo_{p_\ast} & & \dTo_{0}\\
(TX & \rTo & 0) \\
\end{diagram}

The lifting of tangent vectors induced from the isomonodromy flow is a splitting of the map $\pi_{1\ast}$
\begin{diagram}
T_{(X,P,\nabla)}\mathscr{C}onn & & \cong & & \mathbb{H}^1(X,A_P\xrightarrow{[\ ,\hat{s}]}adP\otimes\Omega^1_X) \\
\dTo^{\pi_{1\ast}} \uDashto \\
T_X\mathscr{M}_g & & \cong & & \mathbb{H}^1(X,TX\to 0) \\
\end{diagram}
Notice that the splitting map $s: TX\to A_P$ associated to $\nabla$ gives a map of the complexes 
\begin{diagram}
(A_P & \rTo^{[\ ,\hat{s}]} & adP\otimes\Omega ^1_X) \\
\dTo^{p_\ast} \uDashto_{s} & & \dTo^{0} \uDashto_{0} \\
(TX & \rTo & 0) \\
\end{diagram}
The diagram is commutative because $[\ ,\hat{s}] \circ s$ is basically bracketing $\hat{s}$ with itself and therefore equal to 0. The map of complexes $(s,0)$ is obviously a splitting of the map $(p_\ast,0)$.

The map $(s,0)$ of the complexes of sheaves induce a map on the first hypercohomology, which we denote as $H^1(s)$.
\newtheorem{prop6}{Proposition}[subsection]
\begin{prop6}
\label{prop6}
The isomonodromy lifting $L$ is equal to 
\[
H^1(s): H^1(X,TX)\longrightarrow \mathbb{H}^1(X,A_P\xrightarrow{[\ ,\hat{s}]}adP\otimes\Omega ^1_X)
\]
\end{prop6}
\begin{proof}
At a point ($X$,$P$,$\nabla$) of $\mathscr{C}onn$, let ($X_\epsilon$,$P_\epsilon$,$\nabla _\epsilon$) be an isomonodromy family of triples over $D_\epsilon$ starting with it. Again let $\{ U_i \}_{i\in I}$ be an \v Cech covering of $X$.

Over $U_i$, Let
\[
\tau _{i,\epsilon}: X_\epsilon|_{U_i} \times F \to P_\epsilon|_{U_i}
\]
be the trivialization of $P_\epsilon|_{U_i}$ over $X_\epsilon|_{U_i}$ determined by the flat universal connection (see section~\ref{sec:gfc}) on $P_\epsilon|_{U_i}$, and $\tau _i$ be its restriction at $\epsilon =0$.

Let 
\[
\iota _i: U_i \times D_\epsilon \to X_\epsilon|_{U_i}
\]
be an isomorphism and define 
\[
\phi _i: P|_{U_i}\times D_\epsilon \to P_\epsilon|_{U_i}
\]
as the composition
\[
P|_{U_i} \times D_\epsilon \xrightarrow{(\tau _i^{-1},id_{D_\epsilon})} U_i \times D_\epsilon \times F \xrightarrow{(\iota _i,id_F)} X_\epsilon|_{U_i} \times F \xrightarrow{\tau _{i,\epsilon}} P_\epsilon|_{U_i}
\]

Let $\xi _{ij}$, $\eta _{ij}$, $s_\epsilon$, $s_i$ and $\sigma _i$ be all defined as before in the proofs of proposition~\ref{dp} and section~\ref{sec:dt}. 
Notice that since the local trivializations of the $G$-bundles are canonically given by the flat universal connection, $\tau _{i,\epsilon}$ and $\tau _{j,\epsilon}$ agree on $U_{ij}$, i.e. on $U_{ij}$
\[
\tau _{i,\epsilon}=\tau _{j,\epsilon}
\]
\[
\tau _{i}=\tau _{j}
\]
Therefore over $U_{ij}$, the transition map $\phi _j^{-1} \circ \phi _i$ fits in the diagram
\begin{diagram}
U_{ij} \times F \times D_\epsilon & \rTo^{(\iota _j^{-1} \circ \iota _i,id_F)} & U_{ij} \times F \times D_\epsilon \\
\uTo^{\cong}_{(\tau _i^{-1},id_{D_\epsilon})} & & \uTo^{\cong}_{(\tau _j^{-1},id_{D_\epsilon})} \\
P|_{U_{ij}} \times D_\epsilon & \rTo^{\phi _j^{-1} \circ \phi _i}& P|_{U_{ij}} \times D_\epsilon \\
\end{diagram}
In another word with the local trivializations $(\tau _i^{-1},id_{D_\epsilon})$ and $(\tau _j^{-1},id_{D_\epsilon})$, the transition map $\phi _j^{-1} \circ \phi _i$ corresponds to $(\iota _j^{-1} \circ \iota _i,id_F)$. Let $(\phi _j^{-1} \circ \phi _i)'$ and $\eta _{ij}'$ be $(\phi _j^{-1} \circ \phi _i)$ and $\eta _{ij}$ under the local trivializations, then 
\[
(\phi _j^{-1} \circ \phi _i)' = (\iota _j^{-1} \circ \iota _i,id_F)
\]
and therefore
\[
Id+\epsilon\eta _{ij}' = (Id+\epsilon\xi _{ij},Id_F)
\]
Comparing the coefficients of $\epsilon$ we get
\[
\eta _{ij}'=(\xi _{ij},0)
\]
According to the last paragraph in section~\ref{sec:rtc}, we see this means precisely that $\eta _{ij}=s(\xi _{ij})$.

With $\phi _i:P|_{U_i}\times D_\epsilon \to P_\epsilon|_{U_i}$ defined as above, $s_\epsilon |_{U_i}: TX_\epsilon |_{U_i} \to A_{P_\epsilon}|_{U_i}$ correspond to the section $s_i:TU_i \times D_\epsilon \to A_P|_{U_i} \times D_\epsilon$ constant along $D_\epsilon$, i.e. $\sigma _i=0$.

Therefore $\hat{\sigma} _i=0$, and the pair 
\[
(\eta_{ij})_{i,j\in I},(\hat{\sigma} _i)_{i\in I}
\]
is exactly the hyper 1-cocycle which is the image of $(\xi _{ij},0)$ under the map $H^1(s)$, which finishes the proof.
\end{proof}


\section{extended isomonodromy lifting}

The associated lifting $L_\lambda$ is obtained by extending the isomonodromy lifting $L$ to $\lambda\mathscr{C}onn\to\mathscr{M}_g$ by the $\mathbb{C}^\ast$-action, and multiplying by $\lambda$. For a fixed $\lambda$, $\lambda \neq 0$, the $\mathbb{C}^\ast$-action gives an isomorphism
\[
\mathscr{C}onn  \leftrightarrow  \lambda\mathscr{C}onn
\]
\[
\nabla  \leftrightarrow  \lambda\cdot\nabla
\]
The induced lifting on $\lambda\mathscr{C}onn\to\mathscr{M}_g$ by $L$ via the isomorphism, called the extended isomonodromy lifting, can be written very similarly as $L$. In the same way that the splitting map $s$ associated to a connection $\nabla$ induces a map $H^1(s)$ of hypercohomologies, the $\lambda$-splitting map $s_\lambda$ associated to a $\lambda$-connection $\nabla_\lambda$ induces a map of the corresponding hypercohomology spaces, which will be denoted as $H^1(s_\lambda)$.

\newtheorem{prop7}{Proposition}[section]
\begin{prop7}
The extended isomonodromy lifting of tangent vector on $\lambda\mathscr{C}onn\to\mathscr{M}_g$ is given by:
\[
\frac{1}{\lambda} H^1(s_\lambda): H^1(X,TX)\longrightarrow \mathbb{H}^1(X,A_P\xrightarrow{[\ ,\hat{s}_\lambda]}adP\otimes\Omega^1_X)
\]
\end{prop7}
\begin{proof}
Since the map of moduli spaces is $\nabla \mapsto \lambda\cdot\nabla$ (or $s\mapsto\lambda s$, $\hat{s}\mapsto\lambda\hat{s}$), the induced map on the tangent spaces $T_{(X,P,\nabla)}\mathscr{C}onn \to T_{(X,P,\lambda\nabla)}\lambda\mathscr{C}onn$ is
\[
\mathbb{H}^1(X,A_P\xrightarrow{[\ ,\hat{s}]}adP\otimes\Omega^1_X) \xrightarrow{H^1(id,\lambda )} \mathbb{H}^1(X,A_P\xrightarrow{[\ ,\lambda\hat{s}]}adP\otimes\Omega^1_X)
\]
where $(id,\lambda )$ is the map of complexes of sheaves
\begin{diagram}
(A_P & \rTo^{[\ ,\hat{s}]} & adP\otimes\Omega^1_X) \\
\dTo^{id} & & \dTo_{\lambda} \\
(A_P & \rTo^{[\ ,\lambda\hat{s}]} & adP\otimes\Omega^1_X) \\
\end{diagram}
and $H^1(id,\lambda )$ is the induced map on hypercohomology.

So to get the corresponding lifting on $\lambda\mathscr{C}onn$, i.e. to make the following diagram commutate, the vertical map on the right must be $\frac{1}{\lambda}H^1(\lambda s)$.
\begin{diagram}
\mathbb{H}^1(X,A_P\xrightarrow{[\ ,\hat{s}]}adP\otimes\Omega^1_X) & \rTo^{H^1(id,\lambda )} & \mathbb{H}^1(X,A_P\xrightarrow{[\ ,\lambda\hat{s}]}adP\otimes\Omega^1_X) \\
\uTo^{H^1(s)} & & \uTo^{\frac{1}{\lambda}H^1(\lambda s)} \\
H^1(X,TX) & \rTo{id} & H^1(X,TX) \\
\end{diagram}
\end{proof}

Since $L_\lambda$ is the extended isomonodromy lifting multiplied by $\lambda$, $L_\lambda =H^1(s_\lambda)$. $L_\lambda$ is a $\lambda$-lifting of tangent vectors.


\section{limit lifting at $\lambda =0$}

The continuous limit of $L_\lambda$ at $\lambda =0$ is a 0-lifting $L_0:T_X\mathscr{M}_g\to T_{(X,P,\nabla _0)}\mathscr{H}iggs$. Since $L_\lambda =H^1(s_\lambda)$, by continuity $L_0$ is equal to
\[
H^1(s_0): H^1(X,TX)\longrightarrow \mathbb{H}^1(X,A_P\xrightarrow{[\ ,\hat{s}_0]}adP\otimes\Omega^1_X)
\]
where $s_0$ is the 0-splitting of the Atiyah bundle of $P$ associated to the 0-connection(or Higgs field) $\nabla _0$ on $P$. Because $\pi_{0\ast}\circ H^1(s_0)=0$, so in fact $H^1(s_0)$ can be written as 
\[
H^1(s_0): H^1(X,TX)\longrightarrow \mathbb{H}^1(X,adP\xrightarrow{[\ ,\hat{s}_0]}adP\otimes\Omega^1_X)
\]
The images of a vector in $\vec{t}\in T_X\mathscr{M}_g$ under $H^1(s_0)$ form a vector field on the fiber $Higgs_X$ of $\pi_0$.

Recall that the quadratic Hitchin map on $Higgs_X$ is 
\[
qh: Higgs_X \to H^0(X,\Omega ^{\otimes 2})
\]
\[
(P,s_0) \mapsto <\hat{s}_0,\hat{s}_0>
\]
and its associated lifting of tangent vectors is
\[
L_{qh}: H^1(X,TX) \to \mathbb{H}^1(X,adP\xrightarrow{[\ ,\hat{s}_0]}adP\otimes\Omega^1_X)
\]
\[
f \mapsto H_{qh^\ast f}|_{(P,s_0)}
\] 

The main theorem (Theorem~\ref{thm1:main1}) is that $H^1(s_0)$ is equal to $\frac{1}{2}L_{qh}$. To prove it we need two lemmas. For the first lemma, Let $((\eta _{ij})_{i,j\in I},(\hat{\sigma} _i)_{i\in I})$ be a representative of an arbitrary element $v \in \mathbb{H}^1(X,adP\xrightarrow{[\ ,\hat{s}_0]}adP\otimes\Omega^1_X)$. Because on $U_{ij}$, $<\hat{s}_0,\hat{\sigma}_j-\hat{\sigma}_i>=<\hat{s}_0,[\eta _{ij},\hat{s}_0]>=-<[\hat{s}_0,\hat{s}_0],\eta _{ij}>=0$, therefore 
\[
<\hat{s}_0,\hat{\sigma}_i>=<\hat{s}_0,\hat{\sigma}_j>
\] 
Let $<\hat{s}_0,\hat{\sigma}> \in H^0(X,\Omega ^{\otimes 2})$ denote the resulting global quadratic differential form.

\newtheorem{lem1}{Lemma}[section]
\begin{lem1}
\label{lem1}
Using the above notations, the differential of the map $qh$ is equal to:
\[
qh_\ast : \mathbb{H}^1(X,adP\xrightarrow{[\ ,\hat{s}_0]}adP\otimes\Omega^1_X) \to H^0(X,\Omega ^{\otimes 2})
\]
\[
v \mapsto 2<\hat{s}_0,\hat{\sigma}>
\]
\end{lem1}
\begin{proof}
Let $\{ U_i \}_{i\in I}$ be the \v Cech covering of the curve $X$, $(P_\epsilon,s_\epsilon)$ the family of Higgs bundles over $D_\epsilon$ that correspond to $v$, i.e. for some $\phi _i: P|_{U_i}\times D_\epsilon \to P_\epsilon|_{U_i}$, some $s_i: TU_i \to adP|_{U_i}$ and the given $\sigma _i: TU_i \to adP|_{U_i}$, the diagram
\begin{diagram}
adP|_{U_i} \times D_\epsilon & \rTo^{d\phi _i} & ad{P_\epsilon}|_{U_i} \\
\uDashto^{s_i+\epsilon\sigma _i} \dTo_{p_\ast} &  & \dTo^{p_\ast} \uDashto_{s_\epsilon |_{U_i}} \\
TU_i \times D_\epsilon & \rTo^{id} & TU_i \times D_\epsilon \\
\end{diagram}
is commutative. Because $qh: (P_\epsilon,s_\epsilon) \mapsto <\hat{s}_\epsilon,\hat{s}_\epsilon>$, and that over $U_i$, $<\hat{s}_\epsilon,\hat{s}_\epsilon>=<\hat{s}_i+\epsilon\hat{\sigma}_i,\hat{s}_i+\epsilon\hat{\sigma}_i>=<\hat{s}_i,\hat{s}_i>+2<\hat{s}_i,\hat{\sigma}_i>\epsilon$, so $qh: (P_\epsilon,s_\epsilon) \mapsto <\hat{s}_0,\hat{s}_0>+2<\hat{s}_0,\hat{\sigma}>\epsilon$. Taking the coefficient of $\epsilon$, we see that $qh_\ast$ maps $v$ to $2<\hat{s}_0,\hat{\sigma}>$.
\end{proof}

For the second lemma, let $\omega _H$ be the symplectic 2-form on $Higgs_X$, $((\eta _{ij})_{i,j\in I},(\hat{\sigma} _i)_{i\in I})$ and $((\eta ' _{ij})_{i,j\in I},(\hat{\sigma}' _i)_{i\in I})$ representatives of two vectors $v,v' \in \mathbb{H}^1(X,adP\xrightarrow{[\ ,\hat{s}_0]}adP\otimes\Omega^1_X)$.
\newtheorem{lem2}[lem1]{Lemma}
\begin{lem2}
\label{lem2}
Let $\int : H^1(X,\Omega ^1_X) \to \mathbb{C}$ be the canonical map, then
\[
\omega_H(v,v')=\int (\eta _{ij}\sqcup\hat{\sigma}' _i+\eta ' _{ij}\sqcup\hat{\sigma} _i)
\]
where $\sqcup$ means the cup product $\cup$ of \v Cech cochains composed with the Killing form $<\ ,\ >$.
\end{lem2}
\begin{proof}
see \cite{Ma} Proposition 7.12.
\end{proof}

\newtheorem*{thm}{Theorem}
\begin{thm}
$H^1(s_0)$ is equal to $\frac{1}{2}L_{qh}$.
\end{thm}
\begin{proof}
$\forall f \in H^1(X,TX)$, we want to show that $L_{qh}(f)=2H^1(s_0)(f)$. Let $((\eta ' _{ij})_{i,j\in I},(\hat{\sigma}' _i)_{i\in I})$ be a representative of an element $v \in \mathbb{H}^1(X,adP\xrightarrow{[\ ,\hat{s}_0]}adP\otimes\Omega^1_X)$. 
Using Lemma~\ref{lem1},
\[
\omega _H(L_{qh}(f),v)=d(qh^\ast f)(v)=df(qh_\ast v)=df(2<\hat{s}_0,\sigma >)=f(2<\hat{s}_0,\sigma>)
\]
Using Lemma~\ref{lem2},
\[
\omega _H(H^1(s_0)(f),v)=\omega _H((s_0(f),0),(\eta _{ij},\sigma _i))=f(<\hat{s}_0,\sigma>)
\]
So $L_{qh}(f)=2H^1(s_0)(f)$, $\forall f \in H^1(X,TX)$. Therefore $H^1(s_0)=\frac{1}{2}L_{qh}$.
\end{proof}

\end{document}